\documentclass[11pt, leqno]{article}   
\usepackage[applemac]{inputenc}
\usepackage[T1]{fontenc}
\usepackage{mathtools, bm}
\usepackage{amsthm}
\usepackage{amssymb, bm}
\usepackage{bbm}
\usepackage[mathscr]{euscript}
\usepackage[margin=1in]{geometry}
\usepackage{systeme}
\usepackage[french]{babel}
\usepackage{chngcntr}
\usepackage{titling}
\posttitle{\par\end{center}\vspace{-3mm}}
\usepackage{stmaryrd}
\usepackage{cancel}
\usepackage{fancyhdr,ifthen}
\usepackage[normalem]{ulem} 
\postauthor{\end{tabular}\par\end{center}}
\predate{\begin{center}\large}
\postdate{\par\end{center}\vspace{-20mm}}
\addto\captionsfrench{%
}

\newenvironment{poliabstract}[1]
  {\renewcommand{\abstractname}{#1}\begin{abstract}}
  {\end{abstract}}

\newcommand\blfootnote[1]{%
  \begingroup
  \renewcommand\thefootnote{}\footnote{#1}%
  \addtocounter{footnote}{-1}%
  \endgroup
}

\newtheoremstyle{theoremnoperiod}
  {\topsep}   % ABOVESPACE
  {\topsep}   % BELOWSPACE
  {\normalfont}  % BODYFONT
  {0pt}       % INDENT (empty value is the same as 0pt)
  {\bfseries} % HEADFONT
  {}          % HEADPUNCT
  {5pt plus 1pt minus 1pt} % HEADSPACE
  {}          % CUSTOM-HEAD-SPEC

\newtheorem*{theorem*}{Théorème}
\newtheorem{theorem}{Théorème}[section]

\newtheorem{lemma}[theorem]{Lemme}
\newtheorem*{conjecture*}{Conjecture}

\newtheorem{property}[theorem]{Proposition}

\numberwithin{equation}{section}

\theoremstyle{theoremnoperiod}
\newtheorem*{thmnodot*}{Théorème}

\DeclareMathOperator{\modulo}{mod}
\DeclareMathOperator{\N}{\mathbb{N}}

\DeclareMathOperator{\R}{\mathbb{R}}

\DeclareMathOperator{\C}{\mathbb{C}}

\def\HH{{\mathscr H}}

\DeclareMathOperator{\e}{\rm e}
\def\d{\,{\rm d}}
\def\dd{{\rm d}}
\DeclareMathOperator{\PP}{\mathscr P}

\DeclareMathOperator{\A}{\mathscr A}
\DeclareMathOperator{\J}{\mathscr J}

\DeclareMathOperator{\K}{\mathscr K}

\def\p{{p_{{m},\nu}}}

\newcommand{\sset}{\smallsetminus}
\renewcommand{\leq}{\leqslant}

\renewcommand{\geq}{\geqslant}

\usepackage{color}
\definecolor{vert}{rgb}{0,0.5,0}
\definecolor{violet}{rgb}{0.7,0.1,0.8}

\title{Étude statistique du facteur premier médian, 1 : \goodbreak valeur moyenne}
\author{Jonathan Rotgé \& Gérald Tenenbaum}
\date{}
\begin{document}
\maketitle
\thispagestyle{empty}

\selectlanguage{english}
\begin{poliabstract}{Abstract}We establish an asymptotic formula for the mean value of the logarithm of the middle prime factor of an integer, defined with or without multiplicity, with a relative error of size a negative power of $\log x$, thus refining and generalising a recent result of McNew, Pollack, and Singha Roy.\end{poliabstract}
\vspace{-5mm}
\selectlanguage{french}
\begin{poliabstract}{Résumé}Nous établissons une formule asymptotique pour la valeur moyenne du logarithme du facteur premier médian d'un entier, défini avec ou sans multiplicité, avec un reste relatif de l'ordre d'une puissance négative de $\log x$, précisant et généralisant ainsi un résultat récent de McNew, Pollack et Singha Roy.\end{poliabstract}
\bigskip

%% INTRODUCTION %%

%\blfootnote{2020 {\it Mathematics Subject Classification}: 11N25, 11N37.}
%\blfootnote{{\it Key words and phrases.} middle prime factor, mean value. } 

\section{Introduction et énoncé du résultat}

Pour tout entier naturel $n\geqslant 1$, posons
\[\omega(n):=\sum_{p|n}1,\quad\Omega(n):=\sum_{p^k||n}k,\]
et notons $\nu\in\{\omega,\Omega\}$ l'une ou l'autre de ces fonctions. Si $\{q_j(n)\}_{1\leqslant j\leqslant\omega(n)}$ désigne la suite croissante des facteurs premiers de $n$ comptés sans multiplicité et $\{Q_j(n)\}_{1\leqslant j\leqslant\Omega(n)}$ celle des facteurs premiers de $n$ comptés avec multiplicité, nous écrivons
\[\p(n):=\begin{cases}q_{\lceil\omega(n)/2\rceil}(n)&{\rm si\;} \nu=\omega\\ Q_{\lceil\Omega(n)/2\rceil}(n)&{\rm si\;} \nu=\Omega\end{cases},\quad P^-(n):=q_1(n),\quad P^+(n):=q_{\omega(n)}(n),\]
et convenons que \(p_{m,\nu}(1)=1\).
\par 
Nous dirons qu'un entier $n$ est $y$-friable (resp. $y$-criblé) s'il vérifie $P^+(n)\leqslant y$ (resp. $P^-(n)> y$), avec les convention $P^+(1)=1$, $P^-(1)=\infty$. Dans toute la suite, les lettres $p$ et $q$ désignent des nombres premiers. Définissons
\begin{equation}\label{defHnu}
	\HH_\nu(z):=
	\begin{cases}
		\displaystyle\prod_{q}\Big(1-\frac{1}{q}\Big)^z\Big(1+\frac{z}{q-1}\Big)&(\nu=\omega,\,z\in\C),\\
		\displaystyle\prod_{q}\Big(1-\frac{1}{q}\Big)^z\Big(1-\frac{z}{q}\Big)^{-1}&(\nu=\Omega,\,\Re z<2).
	\end{cases}
\end{equation}
Notons enfin
\begin{equation}\label{defphianu}
	\varphi:=\tfrac12(1+\sqrt 5),\quad A_\nu:=\frac{\varphi\e^{-\gamma}\HH_\nu(\varphi-1)}{\sqrt 5\Gamma(\varphi)}\approx\begin{cases}
		0,493286&\text{si }\nu=\omega,\\
		0{,}414005&\text{si }\nu=\Omega,
	\end{cases}
\end{equation}
où $\gamma$ désigne la constante d'Euler-Mascheroni. Dans un travail récent \cite{pollack}, McNew, Pollack et Singha Roy s'intéressent à la répartition statistique des facteurs premiers intermédiaires à travers l'estimation de diverses quantités. Ils fournissent notamment (\cite[th. 2.6]{pollack}) une estimation de la valeur moyenne de $\log p_{{m},\Omega}(n)$. Posant
\[S_\nu(x):=\sum_{n\leqslant x}\log\p(n)\quad(x\geqslant 3),\]
ils démontrent ainsi que
\[S_{\Omega}(x)=A_{\Omega}x(\log x)^{1/\varphi}\bigg\{1+O\bigg(\frac{(\log_3 x)^{3/2}}{\sqrt{\log_2 x}}\bigg)\bigg\}\quad (x\geqslant 16).\]
Nous nous proposons ici de généraliser et de préciser ce résultat.\par
\begin{theorem}\label{theoreme_principal}
Pour $\nu\in\{\omega,\Omega\}$, nous avons
\begin{equation}\label{estlogpm}
S_\nu(x)=A_\nu x(\log x)^{1/\varphi}
\bigg\{1+O\bigg(\frac1{(\log x)^{1/250}}\bigg)\bigg\}\qquad(x\geq3).
\end{equation}
\end{theorem}
Nous n'avons pas cherché à optimiser l'exposant $1/250$ apparaissant dans \eqref{estlogpm}.

%% FIN INTRODUCTION %%

%%% SECTION 2 %%%

\section{Domaines de contribution principale} 

Notons d'emblée que la contribution à $S_\nu(x)$ des entiers $n$ divisibles par $\p(n)^2$ peut être majorée trivialement. En effet, nous avons
\begin{equation}\label{majo_facteur_carre}
	\sum_{\substack{n\leqslant x\\ \p(n)^2|n}}\log \p(n)\leqslant\sum_{p\leqslant x}\log p\sum_{m\leqslant x/p^2}1\leqslant x\sum_{p\leqslant x}\frac{\log p}{p^2}\ll x.
\end{equation}
Il en est de même de la contribution des entiers $n$ tels que $\p(n)>\sqrt x$. En effet, cette éventualité implique $n=\p(n)$, d'où
\begin{equation}\label{grands_p}
	\sum_{\substack{n\leqslant x\\\p(n)>\sqrt x}}\log\p(n)\leqslant\sum_{\sqrt x<p\leqslant x}\log p\ll x.
\end{equation}\par
Posons
\begin{equation}\label{defPhikxy}
\Phi_{\nu,k}(x,y):=\sum_{\substack{n\leqslant x\\P^-(n)>y\\\nu(n)=k}}1\quad(2\leqslant y\leqslant x,\, k\geqslant 1).
\end{equation}
Si $n\geqslant 2$ désigne un entier tel que $p_{m,\nu}(n)^2\nmid n$, nous l'écrivons sous la forme $n=apb$ où $p=\p(n)$, $a$ est $p$-friable, $b$ est $p$-criblé, et $\nu(b)-\nu(a)\in\{0,1\}$. 
Il suit, à l'aide des estimations \eqref{majo_facteur_carre} et \eqref{grands_p}, 
	\begin{equation}\label{somme_complete}
		\begin{aligned}
			S_\nu(x)&=\sum_{p\leqslant \sqrt x}\log p\sum_{\substack{a\leqslant x/p\\P^+(a)< p}}\sum_{\substack{b\leqslant x/ap\\P^-(b)> p\\\nu(b)-\nu(a)\in\{0,1\}}}1+O(x)\\
			&=\sum_{p\leqslant \sqrt x}\log p\sum_{1\leqslant k\leq(\log x)/\log 4}\sum_{\substack{a\leqslant x/p\\P^+(a)< p\\\nu(a)\in\{k-1,k\}}}\Phi_{\nu,k}\Big(\frac x{ap},p\Big)+O(x).
		\end{aligned}
	\end{equation}
Les quatre énoncés qui suivent nous permettront d'obtenir des estimations de la contribution à \eqref{somme_complete} de différents sous-domaines en $p$ et $k$.

% LEMME 2.1 %

\begin{lemma}\label{borne_Omega}L'estimation
	\begin{equation}\label{majo_grand_Omega}
		\sum_{\substack{n\leqslant x\\\Omega(n)\geqslant k}}1\ll\frac{kx\log x}{2^k},
	\end{equation}
	a lieu uniformément pour $k\geqslant 1$, $x\geqslant 2$.
\end{lemma}
\begin{proof}
	Puisque, pour $k\geqslant 1$, $z\in]1,2[$, nous avons $\textbf{1}_{\{\Omega(n)\geqslant k\}}\leqslant z^{\Omega(n)-k}$, nous obtenons, en appliquant \cite[ex. 66]{tenenbaum_solutions} à la fonction sommatoire de $n\mapsto z^{\Omega(n)}$,
	\begin{align*}
		\sum_{\substack{n\leqslant x\\\Omega(n)\geqslant k}}1\leqslant \sum_{n\leqslant x}z^{\Omega(n)-k}\leqslant xz^{-k}\prod_{p\leqslant x}\frac{p-1}{p-z}\ll \frac{xz^{-k}}{2-z}(\log x)^{z-1}.
	\end{align*}
	Le choix $z=2-1/2k$ fournit alors la majoration
	\[
		\sum_{\substack{n\leqslant x\\\Omega(n)\geqslant k}}1\ll \frac{kx(\log x)^{1-1/2k}}{(2-1/2k)^k}\ll\frac{kx\log x}{2^k}.\qedhere
		\]
\end{proof}
% FIN LEMME 2.1 %

Pour tout ensemble de nombre premiers non vide $E$, posons
\[\omega(n,E):=\sum_{p|n,\,p\in E}1\quad(n\geqslant 1),\quad\Omega(n,E):=\sum_{p^k\|n,\,p\in E}k\quad (n\geqslant 1),\quad E(x):=\sum_{p\leqslant x,\;p\in E}\frac{1}{p}\quad (x\geqslant 2).\]
Dans la suite, nous utiliserons la notation $\nu(n,E)$ pour faire simultanément référence à $\omega(n,E)$ ou $\Omega(n,E)$. Définissons enfin
\[Q(v):=v\log v-v+1\quad(v>0).\]

%% LEMME 2.2%% 

\begin{lemma}\label{hall_tenenbaum_resultat} Soit $p_0$ un nombre premier et $E$ un ensemble de nombres premiers non vide tel que $\min\{p:p\in E\}\geqslant p_0$. Pour tous $0<a<1<b<p_0$, nous avons
 	\begin{equation}\label{Omega_n_E}
 		\sum_{\substack{n\leqslant x\\\nu(n,E)\leqslant a E(x)}}1\ll_{a,p_0} x\e^{-E(x)Q(a)},\quad\sum_{\substack{n\leqslant x\\\nu(n,E)\geqslant b E(x)}}1\ll_{b,p_0} x\e^{-E(x)Q(b)}.
	 \end{equation}
\end{lemma}
\begin{proof}
	Commençons par établir la première majoration \eqref{Omega_n_E}. La méthode de Rankin fournit, pour $0<a<1$,
	\begin{equation}\label{eq:majo:casomega:rankin}
		\sum_{\substack{n\leqslant x\\\nu(n,E)\leqslant a E(x)}}1\leqslant \sum_{\substack{n\leqslant x\\ \omega(n,E)\leqslant aE(x)}}1\leqslant a^{-aE(x)}\sum_{n\leqslant x}a^{\omega(n,E)}.
	\end{equation}
	Par ailleurs, une majoration standard de la moyenne d'une fonction multiplicative comme \cite[th.\thinspace00]{hall_tenenbaum} ou \cite[th.\thinspace III.3.5]{tenenbaum_livre} permet d'écrire
	\begin{equation}\label{eq:majoompowath00}
		\sum_{n\leqslant x}a^{\omega(n,E)}\ll\frac{x}{\log x}\sum_{n\leqslant x}\frac{a^{\omega(n,E)}}{n}\cdot
	\end{equation}
	La première inégalité découle alors de \eqref{eq:majo:casomega:rankin} et \eqref{eq:majoompowath00} en remarquant que
	\[
\sum_{n\leqslant x}\frac{a^{\omega(n,E)}}n
\leqslant\prod_{p\leqslant x,\ p\notin E}\frac1{1-1/p}
\prod_{p\leqslant x,\ p\in E}\bigg(1+\frac a{p-1}\bigg)
\ll_a(\log x)\e^{(a-1)E(x)}.
\]
\par 
	Par ailleurs, en appliquant \cite[ex. 66]{tenenbaum_solutions} à la fonction sommatoire de $n\mapsto b^{\Omega(n,E)}$ nous obtenons
	\[\sum_{\substack{n\leqslant x\\\nu(n,E)\geqslant b E(x)}}1\leqslant xb^{-bE(x)}\prod_{p\leqslant x}\Big(1-\frac1p\Big)\sum_{\nu\geqslant 0}\frac{b^{\Omega(p^\nu,E)}}{p^\nu}= xb^{-bE(x)}\prod_{\substack{p\leqslant x\\ p\in E}}\frac{1-1/p}{1-b/p}\ll_{b,p_0}xe^{-E(x)Q(b)}.\]
	La deuxième majoration \eqref{Omega_n_E} s'ensuit.
\end{proof}

%% FIN LEMME 2.2 %%

%% LEMME 2.3 %%

\begin{lemma}\label{domaines_p_k}
	Notons 
	\[\A_\nu:=\Big\{n\leqslant x:\tfrac14\log_2 x<\nu(n)< 2\log_2 x,\,(\log x)^{3/5}<\log\p(n)\leqslant (\log x)^{16/17}\Big\}.\] 
	Nous avons alors 
	\begin{equation}\label{est-Anu}
		\sum_{n\in[1,x]\smallsetminus\A_\nu}\log\p(n)\ll x(\log x)^{1/\varphi-1/250}.
	\end{equation}
\end{lemma}
\begin{proof}
	Au vu de l'estimation \eqref{grands_p}, il suffit de considérer les entiers $n\leqslant x$ vérifiant $\p(n)\leqslant\sqrt{x}$.\par
	En appliquant \eqref{majo_grand_Omega} avec $k:=\lfloor2\log_2 x\rfloor$, il vient
	\begin{equation}\label{borne_sup_nu}
		\begin{aligned}
			\sum_{\substack{n\leqslant x\\ \nu(n)\geqslant 2\log_2 x}}\log \p(n)&\leqslant \sum_{p\leqslant x}\log p\sum_{\substack{d\leqslant x/p\\\Omega(d)\geqslant\lfloor2\log_2 x\rfloor-1}}1\ll \frac{x\log_2 x}{(\log x)^{\log 4-1}}\sum_{p\leqslant x}\frac{\log p}{p}\\
			&\ll x(\log x)^{2-\log 4}\log_2 x\ll x(\log x)^{1/\varphi-1/250}.
		\end{aligned}
	\end{equation}
	La contribution des entiers $n$ tels que $\nu(n)\geqslant 2\log_2 x$ peut donc être absorbée par le terme d'erreur de~\eqref{estlogpm}. Notons $W_0$ et $W_{-1}$ respectivement les deux branches de la fonction de Lambert, réciproque de la fonction $z\mapsto z\e^z$. Les deux solutions $\xi_0$ et $\xi_1$ de l'équation $1/\varphi+Q(\xi)-1=0$ sont alors
	\begin{align}\label{sol_delta}		\xi_0:=\exp\big\{1+W_{-1}\big(-1/\e\varphi\big)\big\}\approx0.26583,\quad\xi_1:=\exp\big\{1+W_0\big(-1/\e\varphi\big)\big\}\approx1.99374.
	\end{align}
	Cela permet d'estimer la contribution à \eqref{est-Anu} des entiers $n$ tels que $\nu(n)\leqslant \tfrac14\log_2 x$. En effet, notant~$\PP$ l'ensemble des nombres premiers, le théorème de Mertens implique, pour $x$ assez grand et $p\leqslant\sqrt{x}$, 
	\[\PP(x/p)\geqslant \PP(\sqrt{x})\geqslant\log_2 x-1.\]  
	En choisissant $a=\tfrac14$, la première majoration \eqref{Omega_n_E} fournit alors
	\begin{equation}\label{borne_inf_k}
		\sum_{\substack{n\leqslant x,\,p|n\\\nu(n,\PP)\leq(\log_2 x)/4}}1\leqslant\sum_{\substack{d\leqslant x/p\\\omega(d,\PP)\leqslant\{1+\PP(x/p)\}/4}}1\ll \frac{x}{p(\log x)^{Q(1/4)}}\quad(p\leqslant\sqrt{x}).
	\end{equation}
	Nous en déduisons que la contribution à \eqref{est-Anu} de tels entiers est
	\begin{align*}
		\sum_{p\leqslant \sqrt x}\log p\sum_{\substack{n\leqslant x,\,p|n\\\omega(n)\leq(\log_2 x)/4}}1\ll\frac{x}{(\log x)^{Q(1/4)}}\sum_{p\leqslant x}\frac{\log p}{p}\ll x(\log x)^{1-Q(1/4)},
	\end{align*}
	ce qui convient puisque $1-Q(1/4)<1/\varphi-1/250$.\par
	La contribution à \eqref{est-Anu} des entiers $n$ contrevenant à l'encadrement $\tfrac14\log_2 x<\nu(n)\leq2\log_2 x$ est donc acceptable. Si $n\in [1,x]\sset\A_\nu$ vérifie cette condition, nous avons soit $\log\p(n)\leqslant(\log x)^{3/5}$, soit $(\log x)^{16/17}<\log\p(n)\leqslant \tfrac12\log x$. La contribution à \eqref{est-Anu} des entiers correspondant à la première éventualité est
	\begin{equation}\label{borne_inf_p}
		\ll\sum_{\log p\leqslant (\log x)^{3/5}}\frac{x\log p}p\ll x(\log x)^{3/5}.
	\end{equation}
\par 
Considérons les entiers $n$ vérifiant la seconde inégalité. Pour \(y=\exp\{(\log x)^\beta\}\), \(\tfrac{16}{17}\leqslant\beta\leq1\), la condition \(p_{m,\nu}(n)>y\) impose que le nombre de facteurs premiers  \(\leqslant y\) n'excède 
pas celui des facteurs $\geqslant y$. L'indicatrice de ces entiers est donc majorée par la fonction multiplicative
$
4^{\Omega(n,]y,\infty[)-\omega(n,[2,y])}.
$
Une majoration standard telle que \cite[th. III.3.5]{tenenbaum_livre}  pour la valeur moyenne d'une fonction multiplicative positive ou nulle fournit alors
\[
\sum_{n\leqslant x,\ p_{m,\nu}(n)>y}1
\ll x(\log x)^{3-15\beta/4}.
\]
Une sommation d'Abel permet alors de montrer que la contribution correspondante au membre de gauche de \eqref{est-Anu} est $\ll x(\log x)^{8/17}$. Comme $1/\varphi-1/250\approx 0,6140>8/17$, la majoration souhaitée est établie.
\end{proof}
%% FIN LEMME 2.3 %%

Dans la suite, nous nous intéresserons principalement à l'évaluation de la contribution $S_{\nu,\iota}(x)$ à $S_\nu(x)$ des entiers $n$ vérifiant $\nu(n)\equiv 1(\modulo 2)$. Dans ce cas, plus favorable pour les calculs à suivre, nous avons $\nu(a)=k$ dans \eqref{somme_complete}. Nous traiterons en parallèle le cas de la somme complémentaire $S_{\nu,\pi}(x)$ en indiquant au fil des énoncés les éventuelles modifications à prendre en compte. Posons
\begin{equation}\label{S_star}
	\begin{gathered}
		\J_x:=\big]\e^{(\log x)^{3/5}},\e^{(\log x)^{16/17}}\big],\quad \K_x:=\big[\lceil\tfrac18\log_2 x-1\rceil,\lfloor\log_2x\rfloor\big]\cap\R_+^*,\\
		S^*_{\nu,\iota}(x):=\sum_{p\in\J_x}\log p\sum_{k\in\K_x}\sum_{\substack{a\leqslant x/p\\P^+(a)<p\\\nu(a)=k}}\Phi_{\nu,k}\Big(\frac x{ap},p\Big)\quad (x\geqslant 3).
	\end{gathered}
\end{equation}

%% PROPOSITION 2.4 %%
\begin{property} \label{P2.4}
Nous avons l'estimation
	\begin{equation}\label{S_S_star}
		S_{\nu,\iota}(x)=S^*_{\nu,\iota}(x)+O\big(x(\log x)^{1/\varphi-1/250}\big).
	\end{equation}
\end{property}
\begin{proof}
	D'après \eqref{est-Anu}, nous pouvons restreindre la somme intérieure de \eqref{somme_complete} aux entiers $a$ tels que $\tfrac18\log_2 x-1\leqslant\nu(a)=\tfrac12\{\nu(n)-1\}\leqslant \log_2 x$. Ainsi, d'après \eqref{somme_complete}, \eqref{borne_sup_nu} et \eqref{borne_inf_k} nous avons
	\begin{align*}
		S_{\nu,\iota}(x)&=\sum_{p\leqslant \sqrt x}\log p\sum_{k\in\K_x}\sum_{\substack{a\leqslant x/p\\P^+(a)<p\\\nu(a)=k}}\Phi_{\nu,k}\Big(\frac{x}{ap},p\Big)+O\Big(x(\log x)^{1/\varphi-1/250}\Big)\\
		&=S^*_{\nu,\iota}(x)+\sum_{\substack{p\notin\J_x\\ p\leq\sqrt x}}\log p\sum_{k\in\K_x}\sum_{\substack{a\leqslant x/p\\P^+(a)<p\\\nu(a)=k}}\Phi_{\nu,k}\Big(\frac{x}{ap},p\Big)+O\Big(x(\log x)^{1/\varphi-1/250}\Big)\\
		&= S^*_{\nu,\iota}(x)+O\bigg(\sum_{n\in[1,x]\smallsetminus\A_\nu}\log\p(n)+x(\log x)^{1/\varphi-1/250}\bigg).
	\end{align*}
	Une nouvelle application du Lemme \ref{domaines_p_k} fournit alors le résultat annoncé.
\end{proof}
%% FIN PROPOSITION 2.4 %%
%%% FIN SECTION 2 %%%
%%% SECTION 3 %%%

\section{Regroupement}
Réunissons les deux parités avant toute approximation.
Posons,
\begin{equation}\label{definitions_exactes}
\begin{gathered}
h_0(z):=\frac{\e^{-\gamma z}}{\Gamma(1+z)},\quad
F_\nu(z):=\e^{\gamma z}\HH_\nu(z),\quad r_\Omega:=2,\quad r_\omega:=+\infty,\\
I_{k,f}(\xi):=[z^k]\{f(z)\e^{\xi z}\}
=\frac1{2\pi i}\oint_{|z|=r}\frac{f(z)\e^{\xi z}}{z^{k+1}}\d z,\\
u_p:=\frac{\log x}{\log p},\quad\beta_p:=\frac{\log_2p}{\log_2x}.
\end{gathered}
\end{equation}
où $f$ est une fonction holomorphe dans un domaine contenant le disque fermé $|z|\leqslant r$. La fonction $h_0$ est entière et $F_\nu$
est holomorphe sur $|z|<r_\nu$. Notons à fins de référence ultérieure que
\begin{equation}
\label{upb}
\log u_p=(1-\beta_p)\log_2x,\quad \log_2p=\beta_p\log_2x.
\end{equation}

\subsection{Estimations uniformes}
\begin{lemma}\label{borne_coefficients}
 Soient $0<a<b<c$ fixés et $f$ une fonction holomorphe sur $D(0,c)$.
 L'estimation 
\begin{equation}\label{maj-Cauchy}
I_{k,f(\xi)}\ll_{f,a,b}\frac{\xi^k}{k!}
\end{equation}
est valide uniformément pour $k\geqslant 1$, $\xi>0$ et $a\leqslant  k/\xi\leqslant b$.
Si $k\asymp\xi\to\infty$ et $|\eta|\leq1$, les rapports $k/\xi$ et $k/(\xi+\eta)$ variant
 dans un compact admissible, alors
\begin{equation}\label{stabilite_coeffs}
I_{k,f}(\xi+\eta)-I_{k,f}(\xi)
\ll_{f,a,b}|\eta|\frac{\xi^k}{k!}.
\end{equation}
\end{lemma}
\begin{proof}
Sur le cercle de rayon $r=k/\xi$, la formule de Cauchy fournit
\[
I_{k,f}(\xi)\ll r^{-k}\int_{-\pi}^{\pi}\e^{k\cos\vartheta}\d\vartheta
\ll \frac{(\e/r)^k}{\sqrt k}\ll\frac{\xi^k}{k!}\cdot\]
Nous avons utilisé ici la majoration $\cos\vartheta\leqslant 1-2\vartheta^2/\pi^2$ et la formule de Stirling.
Pour la seconde assertion, nous intégrons
$I_{k,f}'(\xi)=I_{k-1,f}(\xi)$ sur le segment réel
d'extrémités $\xi$ et $\xi+\eta$. On conclut grâce à \eqref{maj-Cauchy} et à l'hypothèse concernant $k/\xi$.
\end{proof}

\begin{lemma}\label{Alladi_exact}
Soient $0<a<b$ fixés. Pour $u=(\log x)/\log y$, 
$(k-1)/\log u\in[a,b]$, et sous les conditions
$\exp\{(\log_2x)^3\}\leqslant y\leqslant\sqrt x$, nous avons
\begin{equation}\label{phinuksy;th32}
\Phi_{\nu,k}(x,y)=\frac x{\log x}
\left\{I_{k-1,h_0}(\log u)
+O_{a,b}\left(\frac{(\log u)^{k-1}}{(k-1)!\sqrt u}\right)\right\}.
\end{equation}
\end{lemma}
\begin{proof}
Posons
	\begin{equation}\label{defphinuxyz}
		\Phi_{\nu}(x,y,z):=\sum_{\substack{n\leq x\\P^-(n)>y}}z^{\nu(n)}=\sum_{k\geqslant 0}\Phi_{\nu,k}(x,y)z^k\quad(3\leq y\leq x,\,z\in\C).
	\end{equation}
	Par la formule intégrale de Cauchy, nous avons
	\[\Phi_{\nu,k}(x,y)=\frac1{2\pi i}\oint_{|z|=r}\Phi_{\nu}(x,y,z)\frac{\d z}{z^{k+1}}\cdot\] 
Une application de \cite[th. 3 \& th.4]{alladi}\footnote{Le théorème 3 de \cite{alladi} est énoncé sous la condition supplémentaire $b<2$ dans le cas $\nu=\Omega$, cependant, l'hypothèse de minoration sur $y$ permet d'étendre le résultat à la seule condition que $b$ soit fixé, comme indiqué dans les remarques figurant au paragraphe 8 de \cite{alladi}.}	permet d'écrire, avec $r_{k,y}:=(k-1)/\log u$,
\begin{equation}\label{phinuksy;th32+}
	\begin{aligned}
		\Phi_{\nu,k}(x,y)&=\frac{x}{\log x}\bigg\{\frac{1}{2\pi i}\oint_{|z|=r_{k,y}}h_0(z)\e^{z\log u}\frac{\dd z}{z^{k}}+O\Big(\frac{(\log u)^{k-1}}{(k-1)!\sqrt{u}}\Big)\bigg\}\\
		&=\frac{x}{\log x}\Big\{I_{k-1,h_0}(\log u)+O\Big(\frac{(\log u)^{k-1}}{(k-1)!\sqrt{u}}\Big)\Big\}.
	\end{aligned}
	\end{equation}
\end{proof}
Posons 
\begin{equation}
	\label{deflambdas}
	\lambda_\nu(p,k):=\sum_{\substack{P^+(a)<p\\\nu(a)=k}}\frac1a \quad(3\leq p\leq \sqrt{x},\,k\in{\N}).
	\end{equation}

\begin{lemma}\label{lambda_exacte}
Pour $p\geqslant 3$ et lorsque $k/\log_2p$ est confiné à un compact fixe de $]0,r_\nu[$, nous avons
\begin{equation}\label{evallambdacauchy}
\lambda_\nu(p,k)=I_{k,F_\nu}(\log_2p)
+O\left(\frac{(\log_2p)^k}{k!\log p}\right).
\end{equation}
\end{lemma}
\begin{proof}
Le produit générateur est $\prod_{q<p}(1+z/(q-1))$ pour $\nu=\omega$
et $\prod_{q<p}(1-z/q)^{-1}$ pour $\nu=\Omega$.
Par le théorème des nombres
premiers, nous avons $F_\nu(z)(\log p)^z\{1+O(1/\log p)\}$ uniformément sur les compacts admissibles.
L'extraction sur $|z|=k/\log_2p$, avec la majoration gaussienne
utilisée dans la preuve de \eqref{maj-Cauchy} permet de conclure.
\end{proof}

\subsection{Majoration globale et réduction arithmétique}
Rappelons les notations \eqref{S_star} et \eqref{upb}. Pour $p\in\J_x$, posons $w_p:=\sqrt{\beta_p(1-\beta_p)}\log_2x$ et
\begin{equation}
\label{snu-kappa}
s(p,t):=\frac{w_p^{2t}}{t\Gamma(t)^2} \quad(t\geq 1),\qquad
\kappa(\beta):=\beta-1+2\sqrt{\beta(1-\beta)}\quad(0\leqslant \beta\leqslant 1).
\end{equation}
\begin{lemma}\label{masse}
Nous avons
\begin{equation}\label{maj_coeff_exacts}
\frac{x}{\log x}\sum_{p\in\J_x}\frac{\log p}{p\log u_p}
\sum_{k\in\K_x}s(p,k)\ll x(\log x)^{1/\varphi}.
\end{equation}
\end{lemma}
\begin{proof}
On a $\tfrac4{17}\log_2x\leqslant w_p\leqslant \frac12 \log_2x$ et $\log u_p\asymp\log_2p\asymp \log_2x$.
De la formule de Stirling et des estimations classiques des dérivées de $\log\Gamma$, nous déduisons alors que 
\[
\Big[\frac{\dd \log s(p,t)}{\d t}\Big]_{t=w_p}\ll \frac1{w_p^2},\qquad
\frac{\dd^2\log s(p,t)}{\d t^2}=\frac{-2}t+O\Big(\frac1{t^3}\Big)\leqslant \frac{-1}{3w_p}\quad(t\in\K_x).
\]
Par la formule de Taylor, il suit
\begin{equation}
\label{somspk}
s(p,k)\ll\e^{2w_p-(k-w_p)^2/6w_p}\quad(k\in\K_x),\quad \sum_{k\in\K_x}s(p,k)\ll \e^{2w_p}\sqrt{\log_2x}.
\end{equation}
Posons $W(t):=\sqrt{\beta_t(1-\beta_t)}\log_2x$, de sorte que $W(p)=w_p$. Nous avons $W'(t)\ll1/t\log t$ pour $t\in\J_x$. Une sommation d'Abel utilisant une forme forte du théorème des nombres premiers fournit alors 
\[
\sum_{p\in\J_x}\frac{\log p}{p}\e^{2w_p}
=(\log_2x)\int_{3/5}^{16/17}(\log x)^{\beta+2\sqrt{\beta(1-\beta)}}\d\beta
+O\big(\e^{-(\log x)^{3/10}}\big),
\]
grâce au changement de variables $\log t=(\log x)^\beta$, $\dd t/t=(\log_2x)(\log x)^\beta\d\beta$.
La fonction $\kappa$ définie en \eqref{snu-kappa} possède un maximum strict $1/\varphi$ en
$\beta^*=(5+\sqrt5)/10$; sa dérivée seconde étant strictement négative, on a  $\kappa(\beta)\leqslant 1/\varphi-c_0(\beta-\beta^*)^2$ sur l'intervalle d'intégration pour une constante $c_0>0$ convenable. L'intégrale est $\ll(\log x)^{1+1/\varphi}/\sqrt{\log_2x}$.
On obtient \eqref{maj_coeff_exacts} compte tenu de la seconde estimation \eqref{somspk}.
\end{proof}

Posons
\begin{equation*}
\alpha(k,p):=I_{k-1,h_0}(\log u_p)
\big\{I_{k,F_\nu}(\log_2p)+I_{k-1,F_\nu}(\log_2p)\big\}
\end{equation*}
\begin{property}\label{reduction_exacte}
Pour $x$ assez grand, nous avons
\begin{equation}\label{coeffs_exacts}
\begin{aligned}
S_\nu(x)&=\frac x{\log x}\sum_{p\in\J_x}\frac{\log p}{p}
\sum_{k\in\K_x}\alpha(k,p)+O\Big(x(\log x)^{1/\varphi-1/250}\Big).
\end{aligned}
\end{equation}
\end{property}
\begin{proof}
La proposition \ref{P2.4} et son analogue pair permettent, avec une erreur englobée par celle annoncée en \eqref{coeffs_exacts}, de restreindre la sommation de \eqref{somme_complete} aux conditions $p\in\J_x$, $k\in\K_x$, $\nu(b)=k$, et $\nu(a)=k$ ou $k-1$ selon la parité de $\nu(n)$.
Le même intervalle $\K_x$ convient au cas pair : pour les entiers omis,
 $\nu(n)$ est hors de l'intervalle considéré au Lemme \ref{domaines_p_k}.
Écartons les termes tels que $\Omega(a)>2\log_2x$.
La preuve de \eqref{borne_sup_nu} fournit
\[
\sum_{\substack{n\leqslant x\\\Omega(n)>2\log_2x}}\log p_{m,\nu}(n)
\leqslant\sum_{p\leqslant x}\log p
\sum_{\substack{d\leqslant x/p\\\Omega(d)>2\log_2x-1}}1
\ll x(\log x)^{1/\varphi-1/250}.
\]
Pour les entiers restants, nous avons $\log(ap)\ll (\log_2x)\log p$ et
$ap^3\leqslant p^{2\log_2x+3}<x$. Les hypothèses du Lemme \ref{Alladi_exact} sont satisfaites pour évaluer $\Phi_{\nu,k}(x/ap,p)$: 
\[
\log \frac x{ap}=\log x\Big\{1+O\Big(\frac{\log_2x}{u_p}\Big)\Big\},\quad
v_p:=\frac{\log (x/ap)}{\log p}=u_p\Big\{1+O\Big(\frac{\log_2x}{u_p}\Big)\Big\}.
\]
En particulier, l'hypothèse $p\geqslant\exp\{(\log_2(x/ap))^3\}$ est bien satisfaite pour $x$ assez grand et  
le rapport $(k-1)/\log v_p$ parcourt un compact admissible.  En appliquant alors le Lemme \ref{Alladi_exact} et l'estimation \eqref{stabilite_coeffs}, nous obtenons
\[
\Phi_{\nu,k}\Big(\frac x{ap},p\Big)=\frac{x}{ap\log x}
\left\{I_{k-1,h_0}(\log u_p)
+O\left(\frac{(\log u_p)^{k-1}}{(k-1)!}
\left\{\frac1{\sqrt{u_p}}+\frac{\log_2x}{u_p}\right\}\right)\right\}.
\]
L'estimation \eqref{stabilite_coeffs} a été utilisée ici pour tenir compte de la différence $\log v_p-\log u_p$. La condition $a\leqslant x/p$ est automatique lorsque $\Omega(a)\leq2\log_2x$.
Pour $\nu=\Omega$, la restriction restante est déjà satisfaite.
Pour $\nu=\omega$, sa suppression dans la somme harmonique implique une erreur,
pour $j=k$ ou $k-1$,
\[
\sum_{\substack{P^+(a)<p,\ \omega(a)=j\\\Omega(a)>2\log_2x}}\frac1a
\ll\sqrt{j+1}(\tfrac32)^{j-2\log_2x}\frac{(\log_2p)^j}{j!}\cdot
\]
Cette estimation est obtenue par la méthode de Rankin et nous omettons les détails.
\par 
Insérons \eqref{evallambdacauchy} pour évaluer la somme en $a$ restante. En remplaçant les valeurs de $\Gamma$ par des factorielles, on constate que les facteurs qui majorent les erreurs pour les deux parités s'expriment exactement en fonction de $s(p,k)$ comme suit \[
\frac{(\log u_p)^{k-1}(\log_2p)^k}{(k-1)!k!}=\frac{s(p,k)}{\log u_p},\qquad
\frac{(\log u_p)^{k-1}(\log_2p)^{k-1}}{(k-1)!^2}
=\frac{ks(p,k)}{(\log_2p)\log u_p}\cdot
\]
Comme $k/\log_2p\asymp1$, les erreurs restantes sont majorées par le produit du membre de gauche de \eqref{maj_coeff_exacts} par une quantité
\[
\ll \frac1{(\log x)^{1/34}}+\frac{\log_2x}{(\log x)^{1/17}}
+\frac1{(\log x)^{3/5}}+\frac{\sqrt{\log_2x}}{(\log x)^{\log(3/2)}}\ll\frac1{(\log x)^{1/250}}\cdot
\]
L'estimation annoncée en découle.
\end{proof}

\section{Preuve du Théorème \ref{theoreme_principal}}
Posons $U_\beta=(1-\beta)\log_2x$, $V_\beta=\beta \log_2x$, de sorte que $U_{\beta_p}=\log u_p$, $V_{\beta_p}=\log_2p$. Posant 
\[
T_x(\beta):=\sum_{k\in\K_x}I_{k-1,h_0}(U_\beta)
\{I_{k,F_\nu}(V_\beta)+I_{k-1,F_\nu}(V_\beta)\},
\]
nous avons alors, en vertu de \eqref{coeffs_exacts},
$$S_\nu(x)={x\over \log x}\sum_{p\in\J_x}\frac{\log p}pT_x(\beta_p)+O\Big(x(\log x)^{1/\varphi-1/250}\Big).$$
La formule de Cauchy appliquée sur des cercles fixes admissibles fournit
$|T_x(\beta)|+|T_x'(\beta)|\ll(\log x)^C$ sur l'intervalle principal.
Puisque $\beta_t'=1/\{t(\log t)\log_2x\}$, nous déduisons d'une intégration par parties et du théorème des nombres premiers,
via le changement de variables $t=\exp\{(\log x)^\beta\}$, que
\begin{equation}\label{coefficients_exacts_integrale}
S_\nu(x)=\frac{x\log_2x}{\log x}\int_{3/5}^{16/17}(\log x)^\beta
T_x(\beta)\d\beta+O\big(x(\log x)^{1/\varphi-1/250}\big).
\end{equation}
L'erreur issue du théorème des nombres premiers est
$\ll x\e^{-(\log x)^{3/10}}$.
\medskip\par 
\noindent{\it Suppression des restrictions}.
Pour $a>0$, $0<b<r_\nu$, $k\geq0$ et $j\in\{0,1\}$, nous pouvons déduire de la formule de Cauchy que
\begin{equation}
\label{expreste}
|I_{k,h_0}(U_\beta)I_{k+j,F_\nu}(V_\beta)|
\ll_{a,b}(ab)^{-k}b^{-j}\e^{aU_\beta+bV_\beta}.
\end{equation}
Pour  $\frac35\leqslant \beta\leqslant \frac{16}{17}$, posons $w_\beta=\sqrt{U_\beta V_\beta}$ et
$d_\beta=\sqrt{V_\beta/U_\beta}$.
\par Si $k\leqslant \frac18 \log_2x+O(1)$, choisissons $(a,b)=\tfrac9{10}(d_\beta,1/d_\beta)$;
si $k\geqslant \log_2x+O(1)$, choisissons $(a,b)=\frac{11}{10}(d_\beta,1/d_\beta)$.
Ces couples parcourent  des compacts admissibles. L'exposant global de \eqref{expreste} est de la forme $2w_\beta-\Delta$ avec 
$$\Delta\geqslant \big\{\tfrac{8}{170}-\tfrac14\log(\tfrac{10}9)\big\}\log_2x+O(1)$$ dans le premier cas et
$$\Delta\geqslant \big\{2\log(\tfrac{11}{10})-\tfrac1{10}\big\}\log_2x +O(1)$$ dans le deuxième. 
Ces minorations sont toutes deux $\geqslant \tfrac1{100}\log_2x$ pour $x$ assez grand.
La sommation géométrique fournit
\[
\sum_{\substack{k\geq0\\k+1\notin\K_x}}|I_{k,h_0}(U_\beta)|
\{|I_{k+1,F_\nu}(V_\beta)|+|I_{k,F_\nu}(V_\beta)|\}
\ll\frac{\e^{2w_\beta}}{(\log x)^{1/100}}.
\]
Après intégration et multiplication par $x(\log x)^{\beta-1}\log_2x$,
l'erreur impliquée est $$\ll x(\log x)^{1/\varphi-1/100}\log_2x,$$ donc acceptable.
La série complète satisfait, pour $r>1/r_\nu$ (avec $1/r_\omega=0$),
\begin{equation}\label{identite_coefficients_exacts}
\begin{aligned}
&\sum_{k\geq0}I_{k,h_0}(U_\beta)
\{I_{k+1,F_\nu}(V_\beta)+I_{k,F_\nu}(V_\beta)\}=\frac1{2\pi i}\oint_{|z|=r}
\frac{1+z}{z}h_0(z)F_\nu(1/z)
\e^{U_\beta z+V_\beta/z}\d z.
\end{aligned}
\end{equation}
Cette identité coïncide avec celle du produit des séries de Taylor, absolument et uniformément convergentes
sur le cercle. Le terme $1/z$ sélectionne les indices
égaux et le terme $1$ les indices décalés de $1$.
Les coefficients sont définis par les séries également aux extrémités
$\beta=0,1$.
Pour $0\leqslant \beta\leqslant \tfrac35$, le rayon $r=\tfrac32$ induit, après multiplication
par $(\log x)^{\beta-1}$, l'exposant $\tfrac12+\tfrac16\beta\leqslant\tfrac35$.
Pour $\tfrac{16}{17}\leqslant \beta\leq1$, le rayon $r=4$ fournit
$3-\tfrac{11}4\beta\leqslant \tfrac7{17}$.
La contribution globale de ces intervalles 
est $\ll x\{(\log x)^{3/5}+(\log x)^{7/17}\}\log_2x$.
Comme $1/\varphi-\tfrac35>\tfrac1{250}$, elle est absorbée par le reste.
On peut donc remplacer dans \eqref{coefficients_exacts_integrale}  l'intervalle d'intégration par $[0,1]$ et $T_x$
par la série complète .

\par \medskip
\noindent{\it Intégration et résidu.}
Pour le choix $r=\tfrac32$, l'intégration en $\beta$ fournit
\begin{equation}\label{integrale_contour_finale}
\begin{aligned}
S_\nu(x)&=\frac{x}{2\pi i}\oint_{|z|=3/2}
\frac{(1+z)h_0(z)F_\nu(1/z)}{1+z-z^2}
\Big\{(\log x)^{1/z}-(\log x)^{z-1}\Big\}\d z\\
&\quad+O\Big(x(\log x)^{1/\varphi-1/250}\Big).
\end{aligned}
\end{equation}
 L'interversion de l'ordre d'intégration est justifiée par la convergence uniforme.
L'intégrande complet possède des singularités éliminables aux zéros de
$1+z-z^2$, puisque $1/z=z-1$ en ces points.
Séparons les termes : celui qui est associé à $(\log x)^{z-1}$
est $\ll x\sqrt{\log x}$.
Pour celui qui implique $(\log x)^{1/z}$, déplaçons le cercle
$|z|=\tfrac32$ vers $|z|=\tfrac74$.
L'anneau ne contient que le pôle $z=\varphi$; $F_\nu(1/z)$
y est holomorphe. L'autre pôle $-1/\varphi$ est à l'intérieur
des deux cercles. L'intégrale extérieure est $\ll x(\log x)^{4/7}$.
Le résidu en $z=\varphi$ vaut
\[
-\frac{1+\varphi}{\sqrt5}h_0(\varphi)F_\nu(1/\varphi)(\log x)^{1/\varphi}.
\]
Il doit être affecté d'un coefficient $-1$ pour tenir compte de la relation entre les intégrales intérieure et extérieure.  La contribution principale vaut donc
\[
\frac{1+\varphi}{\sqrt5}h_0(\varphi)F_\nu(1/\varphi)x(\log x)^{1/\varphi}
=A_\nu x(\log x)^{1/\varphi},
\]
car $1/\varphi=\varphi-1$ et
$\Gamma(1+\varphi)=\varphi\Gamma(\varphi)$.
Puisque $1/\varphi-\tfrac47>\tfrac1{250}$, les intégrales résiduelles sont englobées
dans le reste. Le domaine borné de $x$ est absorbé dans la constante
du terme d'erreur. Cela achève la preuve.\qed

\bigskip
{\footnotesize
\noindent\textbf{Jonathan Rotgé}\par
\noindent Université d'Aix-Marseille,
Institut de Mathématiques de Marseille,
CNRS UMR 7373\par
\noindent 163 avenue de Luminy, Case 907,
13288 Marseille Cedex 9, France\par
\noindent\textit{E-mail} :
\texttt{jonathan.rotge@etu.univ-amu.fr}\par

\medskip
\noindent\textbf{Gérald Tenenbaum}\par
\noindent Institut \'Elie Cartan, Universit\'e de Lorraine, BP 70239,
54506 Vand\oe uvre-lès-Nancy Cedex, France\par 
\noindent\textit{E-mail} :
\texttt{gerald.tenenbaum@univ-lorraine.fr}\par
}
\end{document}